\documentclass[11pt,leqno]{amsart}
\usepackage{amsmath,amssymb,amsthm}
\usepackage{hyperref}
\newtheorem{theorem}{Theorem}[section]

\newtheorem{proposition}[theorem]{Proposition}

\theoremstyle{definition}
\newtheorem{definition}[theorem]{Definition}

\theoremstyle{remark}

\numberwithin{equation}{section}

\def\fnote#1{\footnote}

\def\ignora#1{}
\def\n3#1{\left\vert  \! \left\vert \! \left\vert \, #1 \, \right\vert \!
  \right\vert \! \right\vert }

\begin{document}

\title{ Diameter two properties and polyhedrality }

\author{ Gin{\'e}s L{\'o}pez-P{\'e}rez}\thanks{The research of G. L\'opez-P\'erez was partially supported by MINECO (Spain) Grant MTM2015-65020-P and Junta de Andaluc\'{\i}a Grant FQM-185.}
\address[G. L\'opez-P\'erez]{Universidad de Granada, Facultad de Ciencias.
Departamento de An\'{a}lisis Matem\'{a}tico 18071-Granada
(Spain) and Instituto de Matem\'aticas de la Universidad de Granada (IEMath-GR)} \email{glopezp@ugr.es}
\urladdr{\url{http://wdb.ugr.es/local/glopezp}}

\author{ Abraham Rueda Zoca }\thanks{The research of A.~Rueda Zoca was supported by a research grant Contratos predoctorales FPU del Plan Propio del Vicerrectorado de Investigaci\'on y Transferencia de la Universidad de Granada, by MINECO (Spain) Grant MTM2015-65020-P and by Junta de Andaluc\'ia Grants FQM-0185.}
\address[A. Rueda Zoca]{Universidad de Granada, Facultad de Ciencias.
Departamento de An\'{a}lisis Matem\'{a}tico, 18071-Granada
(Spain)} \email{ abrahamrueda@ugr.es}
\urladdr{\url{https://arzenglish.wordpress.com}}

\maketitle

\begin{abstract}
We give two examples of polyhedral Banach spaces failing all the diameter two properties, showing that there is not any connection between polyhedrality and the diameter two properties.
\end{abstract}

\bigskip

\section{Introduction}\label{sectintro}

\bigskip

Recall that a Banach space is said to have the \textit{slice diameter two property} (slice-D2P) (respectively \textit{diameter two property} (D2P), \textit{strong diameter two property} (SD2P)) if every slice (respectively non-empty relatively weakly open subset, convex combination of slices) of the unit ball has diameter two. These three geometrical properties, which are extremely opposite to the isomorphic ones given by the Radon-Nikodym property (respectively convex point of continuity property, strong regularity), have shown to be different in a extreme way \cite{eje1, eje2}.

It is a natural question how diameter two properties affect to the geometry of the unit ball of a Banach space. In this setting, very recent results have appeared (see e.g. \cite{alnt}) dealing with the problem of how convex and smooth can a Banach space with any diameter two property be. For instance, there exists an example of Banach space enjoying the D2P and being \textit{midpoint locally uniformly rotund} and it is known that the bidual of a Banach space with the SD2P can not be neither smooth nor strictly convex (see \cite{alnt} and references therein).

In this paper we will continue exploring this connection in an opposite way. Roughly speaking, we wonder if a Banach space whose unit sphere is formed by big faces can fail the diameter two properties. This face structure can be encoded by the concept of \textit{polyhedral Banach spaces}.

According to \cite{fove}, we shall give the following definition.

\begin{definition}\label{defipoly}
Let $X$ be a Banach space.
\begin{enumerate}
\item\label{defi1} We will say that $X$ is \textit{I-polyhedral} if 
$$(Ext B_{X^*})'\subseteq\{0\},$$
where the last accumulation is considered on the weak-star topology of $X^*$ and $Ext(B_{X^*})$ denotes the set of all extreme points of $B_{X^*}$.

\item\label{defi2} We will say that a Banach space $X$ is \textit{II-polyhedral} if
$$(Ext B_{X^*})'\subseteq rB_{X^*},$$
for some $0<r<1$, where the last accumulation is considered on the weak-star topology of $X^*$.

\item\label{defi3} We will say that a Banach space $X$ is \textit{III-polyhedral} if
$$(Ext B_{X^*})'\subseteq int(B_{X^*}),$$
where the last accumulation is considered on the weak-star topology of $X^*$.

\item\label{defi4} We will say that $X$ is \textit{IV-polyhedral} if $f(x)<1$ whenever $x\in S_X$ and $f\in Ext(B_{X^*})'$.

\item\label{defi5} We will say that $X$ is \textit{V-polyhedral} if $\sup\{f(x): f\in Ext(B_{X^*})\setminus D(x)\}<1$ for each $x\in S_X$ where
$$D(x):=\{f\in S_{X^*}: f(x)=1\}.$$

\item\label{defi6} We will say that $X$ is \textit{VI-polyhedral} if each $x\in S_X$ has a neighborhood $V$ satisfying
$$y\in V\cap S_X\Rightarrow [x,y]\subseteq V.$$

\item\label{defi7} We will say that $X$ is VII-polyhedral if the set
$$\{x\in S_X: \max \langle D(x),v\rangle\leq 0\}$$
is open in $S_X$ for each $v\in S_X$, where
$$\langle D(x),v\rangle:=\{f(v): f\in D(x)\}\subseteq \mathbb R.$$

\item\label{defi8} We will say that $X$ is K-polyhedral if the unit ball of every finite-dimensional subspace is a polytope, that is, the convex hull of a finite set.
\end{enumerate}
\end{definition}
It is known that each of the previous polyhedrality notions implies the following one and that no reverse implication holds  \cite{fove}. It is also known that if $X$ is a $I$-polyhedral Banach space then it is isometric to a subspace of $c_0(\Gamma)$, for $\Gamma=dens(X)$ \cite[Theorem 1.2]{fove}. Since every polyhedrality condition is inherited by closed subspaces then every infinite-dimensional I-polyhedral Banach space is a non-reflexive $M$-embedded Banach space \cite[Example III.1.4]{hww} and, as a consequence of the proof of \cite[Corollary 2.5]{lop}, these spaces have the SD2P. Summarising, I-polyhedral Banach spaces enjoy the SD2P. However, it is not clear at all whether the rest of notions of polyhedrality imply any diameter two property.

The aim of this note is to provide examples of polyhedral Banach spaces failing the diameter two properties. Indeed, by giving a slight modification of \cite[Example 4.1]{fove}, we prove in Theorem \ref{ejeIIpolysinstron} that, for each $\varepsilon>0$, there exists a II-polyhedral Banach space whose unit ball contains slices of diameter smaller than $\varepsilon$, proving in Proposition \ref{nodentaII} that this theorem can not be improved. We will also prove in Proposition \ref{ejeVIIpolydif} that there are VII-polyhedral Banach space whose dual unit ball contain a point of Fr\'echet differentiability by analysing \cite[Example 4.5]{fove}. This result proves that K-polyhedrality, which is the most natural and classical generalisation of polyhedrality to an infinite-dimensional framework, is not related at all with the diameter two properties.

\textbf{Notation:} We will consider only real Banach spaces. Given a Banach space $X$, we will denote the closed unit ball and the unit sphere by $B_X$ and $S_X$ respectively. We will also denote by $X^*$ the topological dual of $X$. By a slice of a bounded subset $C$ of $X$ we will mean a set of the following form:
$$S(C,f,\alpha):=\{x\in C: f(x)>\sup f(C)-\alpha\},$$
where $f\in X^*$ and $\alpha>0$. When $X$ is itself a dual Banach space, the previous set will be a $w^*$-slice if $f$ belongs to the predual of $X$.

Given a Banach space $X$ and a point $x\in X$, we say that $x$ is a \textit{point of Fr\'echet differentiability} if, for every $h\in X$, the following limit exists
$$\lim\limits_{t\rightarrow 0}\frac{\Vert x+th\Vert-\Vert x\Vert}{t}$$ 
and it is uniform for $h\in S_X$. It is known that $x$ is a point of Fr\'echet differentiabity of $X$ if, and only if, $\inf\limits_{\alpha>0}diam(S(B_{X^*},x,\alpha))=0$ \cite[Lemma 8.4]{fhhmpz}. Finally, we refer to \cite{hww} for a a detailed treatment about theory of $M$-embedded spaces.

\section{Main results}\label{sectmain}
\bigskip

As we have seen in Section \ref{sectintro}, infinite-dimensional I-polyhedral Banach space are non-reflexive $M$-embedded Banach spaces and, as a consequence, this class of spaces enjoys the SD2P. However, it is not so clear whether a similar statement can be formulated in relation to the rest of notions of polyhedrality. Indeed, the next theorem proves that the unit ball of a II-polyhedral Banach space can even have slices of small diameter.

\begin{theorem}\label{ejeIIpolysinstron}
For every $\varepsilon>0$ there exists a II-polyhedral Banach space $X$ such that $B_X$ contains slices of diameter smaller than $\varepsilon$.
\end{theorem}

\begin{proof}
Pick $\varepsilon>0$ and choose $0<r<\frac{\varepsilon}{2}$. Define
$$U^*:=co\left(B_{(\ell_1\oplus_\infty\mathbb R)^*}\cup\left\{\left(0,1+r\right),\left(0,-1-r\right)\right\}\right),$$
which is clearly a weak$^*$ compact set in $(c_0\oplus_1\mathbb R)^*$. Consequently, there is a norm  $|||\cdot|||$ on $c_0\oplus \mathbb R$ whose unit ball is
$$U:=\{(x,\beta)\in c_0\oplus \mathbb R: \phi(x,\beta)\leq 1\ \mbox{ for all } \phi\in U^*\}.$$
Consider $X:=(c_0\oplus\mathbb R, |||\cdot|||)$, and let us prove that $X$ satisfies the desired requirements. First of all, the same argument of \cite[Example 4.1]{fove} shows that $X$ is a II-polyhedral Banach space. Indeed, it is clear that
$$Ext(U^*)=\left\{\left(0,\pm\left(1+r\right) \right)\right\}\cup \{(\xi e_n,\psi 1): n\in\mathbb N\mbox{ and } \xi,\psi\in\{-1,1\}\}.$$
Therefore, $(0,0)\notin Ext(U^*)'=\{(0,\pm 1)\}\subseteq \frac{1}{1+r}U^*$. Let us now prove that $B_X$ contains a slice of diameter smaller than $\varepsilon$. For this, notice that $B_{\ell_1\oplus_\infty\mathbb R}\subseteq U^*\subseteq (1+r)B_{\ell_1\oplus_\infty\mathbb R}$, so
$$\frac{1}{1+r}B_{c_0\oplus_1\mathbb R}\subseteq U\subseteq B_{c_0\oplus_1\mathbb R}.$$
Consequently, for each pair $(x,\beta)\in X$, it follows
\begin{equation}\label{desiteonorma}
\Vert (x,\beta)\Vert_1\leq |||(x,\beta)|||\leq (1+r)\Vert (x,\beta)\Vert_1.
\end{equation}
Now consider $\delta>0$ such that
\begin{equation}\label{desiteodelta}
2r+3\delta<\varepsilon.
\end{equation}
Define $S:=S(B_X,(0,1+r),\delta)$, and let us prove that $diam(S)<\varepsilon$. To this aim pick $(x,\beta)\in S$. Then
$$1\geq (1+r)\beta=(0,1+r)(x,\beta)>1-\delta\Rightarrow \frac{1}{1+r}\geq \beta\geq \frac{1-\delta}{1+r}.$$
The previous inequality joint to (\ref{desiteonorma}) yields
$$1\geq |||(x,\beta)|||\geq \Vert (x,\beta)\Vert_1=\Vert x\Vert+\vert \beta\vert>\Vert x\Vert+\frac{1-\delta}{1+r},$$
so $\Vert x\Vert\leq \frac{r+\delta}{1+r}$. Finally, given $(y,\gamma)\in S$, we have
$$|||((x,\beta)-(y,\gamma)|||\leq (1+r)\Vert (x,\beta)-(y,\gamma)\Vert_1=(1+r)(\Vert x-y\Vert+\vert \beta-\gamma\vert)$$
$$\leq 2(r+\delta)+\delta=2r+3\delta\mathop{\leq}\limits^{\mbox{(\ref{desiteodelta})}}\varepsilon.$$
From the arbitrariness of $(x,\beta),(y,\gamma)\in S$ we conclude that $diam(S)\leq \varepsilon$, so we are done.\end{proof}

In view of the above theorem, a natural question is whether the unit ball of a II-polyhedral Banach space can be dentable. The answer is negative, which proves that the conclusion of Theorem \ref{ejeIIpolysinstron} is sharp.
\begin{proposition}\label{nodentaII}
Let $X$ be an infinite-dimensional II-polyhedral Banach space. Then $B_X$ does not have slices of arbitrarily small diameter.
\end{proposition}

\begin{proof}
Since $X$ is II-polyhedral there exists $r<1$ such that
\begin{equation}\label{condiextreprop}
(Ext(B_{X^*})'\subseteq rB_X.
\end{equation}
Let us prove that each slice of $B_X$ has diameter, at least, $2(1-r)$. To this aim pick a slice $S:=S(B_X,g,\alpha)$ and a point $x\in S$. Notice that (\ref{condiextreprop}) joint to the weak-star compactness of $B_{X^*}$ implies that the set
$$A:=\{f\in Ext(B_{X^*}):f(x)>r\}$$
is finite. Hence, we can consider $y\in \bigcap\limits_{f\in A}ker(f)\cap ker(g)\cap S_X$. Let us prove that $x\pm (1-r)y\in S$. On the one hand, we have
$$g(x\pm(1-r)y)=g(x)>1-\alpha$$
because $x\in S$. On the other hand, we have
$$\Vert x\pm (1-r)y\Vert=\sup\limits_{f\in Ext(B_{X^*})}f(x)\pm (1-r)f(y)=$$
$$\max\left\{\sup\limits_{f\in A} f(x)\pm (1-r)f(y), \sup\limits_{f\in Ext(B_{X^*})\setminus A} f(x)\pm (1-r)f(y) \right\}.$$
From the definition of $A$ and the assumptions on $y$ we deduce that the previous estimate is smaller than or equal to $1$. Summarising we have proved that $x\pm(1-r)y\in S$, so $diam(S)\geq 2(1-r)$, as desired.\end{proof}
The next proposition proves that we can even go further when considering a more relaxed notion of polyhedrality.
\begin{proposition}\label{ejeVIIpolydif}
There exists a VII-polyhedral Banach space whose dual unit ball contains points of Fr\'echet differentiability.
\end{proposition}

\begin{proof}
Let $\{\omega_n\}_{n\geq 2}$ be a sequence in $\left]\frac{5}{6},1\right]$ such that $\{\omega_n\}\rightarrow 1$. Define $X:=(c_0,|||\cdot|||)$, where the norm $|||\cdot|||$ is defined by
$$|||x|||:=\max\left\{\max\limits_{n\geq 2} \vert x(n)\vert, \max\limits_{n\geq 2} \vert x(1)\vert + \frac{1}{3}\vert x(n)\vert, \sup\limits_{n\geq 2} \omega_n\vert x(1)\vert+\frac{1}{2}\vert x(n)\vert \right\}.$$
Then $X$ is an example of VII-polyhedral Banach space which is not a VI-polyhedral Banach space \cite[Example 4.5]{fove}. Our aim is to show that $B_{X^*}$ has a point of Fr\'echet differentiability. Note that, from the definition of the norm of $X$, it follows that $e_1\in S_X$ and $e_1^*\in S_{X^*}$. Pick $\varepsilon>0$ and consider
$$S_\varepsilon=S(B_X,e_1^*,\varepsilon).$$
Given $x,y\in S_\varepsilon$ we have that
$$1-\varepsilon<x(1)\Rightarrow 1\geq |||x|||\geq \max\limits_{n\geq 2} \vert x(1)\vert+ \frac{1}{3}\vert x(n)\vert\geq 1-\varepsilon+\frac{1}{3}\max\limits_{n\geq 2} \vert x(n)\vert.$$
Thus $\max\limits_{n\geq 2} \vert x(n)\vert\leq 3\varepsilon$ and $\max\limits_{n\geq 2} \vert y(n)\vert\leq 3\varepsilon.$
Hence $\vert x(1)-y(1)\vert\leq \varepsilon$ and $\max\limits_{n\geq 2}\vert x(n)-y(n)\vert\leq 6\varepsilon$. Keeping in mind the above estimate we get
$$||| x-y|||\leq \max\{\varepsilon,3\varepsilon,
4\varepsilon \}=4\varepsilon.$$
Hence $diam(S(B_X,e_1^*,\varepsilon))\leq 4\varepsilon$, and thus
$$\lim\limits_{\varepsilon\rightarrow 0} diam(S(B_X,e_1^*,\varepsilon))=0.$$
Consequently, $B_{X^*}$ contains a point of Fr\'echet differentiability, so we are done.
\end{proof}

\end{document}